\documentclass[10pt,reqno]{amsart}
\usepackage{amsmath}
\usepackage{amssymb}
\usepackage{amsthm}
\usepackage{mathabx}
\usepackage{mathrsfs}
\usepackage{enumitem}
\usepackage{bbm}
\usepackage{times}
\usepackage[colorinlistoftodos]{todonotes}
\usepackage[colorlinks=true,linkcolor=blue,citecolor=blue,linktocpage=true,urlcolor=black]{hyperref} 
\usepackage[abbrev,nobysame]{amsrefs}

\usepackage[capitalize]{cleveref}

\usepackage{caption} 

\usepackage{makecell}
\usepackage{float}
\restylefloat{table}

\usepackage{tablefootnote}

\usepackage{color}
\newtheorem{theorem}{Theorem}[section]
\newtheorem{proposition}[theorem]{Proposition}

\newtheorem{lemma}[theorem]{Lemma}
\newtheorem{corollary}[theorem]{Corollary}
 
\theoremstyle{definition}
\newtheorem{definition}[theorem]{Definition}

\newtheorem{conjecture}[theorem]{Conjecture}

\newtheorem{remark}[theorem]{Remark}

\newcommand{\supp}{\textrm{supp}}
\begin{document}
    
\title[Furstenberg measure of infinite covolume groups in higher rank]{Singularity of Furstenberg measure for infinite covolume discrete subroups in higher rank}
\date{\today}

\author[]{HOMIN LEE}
\email{hominlee@kias.re.kr}
\address{Center for Mathematical Challenges, Korea Institute for Advanced Study, 85 Hoegiro Dongdaemun-gu, Seoul, Republic of Korea}
\author[]{GIULIO TIOZZO}\email{tiozzo@math.utoronto.ca}
\address{University of Toronto, 40 St George St, Toronto ON, Canada}
\author[]{WOUTER VAN LIMBEEK}\address{Department of Mathematics, Statistics, and Computer Science \\ University of Illinois at Chicago \\ Chicago, IL, USA}


\begin{abstract} We consider symmetric random walks on discrete, Zariski-dense subgroups $\Gamma$  of a semisimple Lie group $G$ with Property (T). We prove that if $\Gamma$ has infinite covolume, then the associated hitting measure on the Furstenberg boundary of $G$ is singular. This is in contrast to Furstenberg's discretization of Brownian motion to lattices, and it is the first result of this type when $G$ has higher rank.
\end{abstract}

\maketitle

\section{Introduction}
Let $G$ be a semisimple Lie group with no compact factors, and fix a probability measure $\mu$ on $G$ with countable support such that the semigroup $\Gamma=\Gamma_\mu$ generated by the support of $\mu$ is Zariski-dense. In this setting, by work of Furstenberg, Guivarc'h-Raugi and Gol'dsheid-Margulis 
\cites{Furstenberg, Guivarch-Raugi, Goldsheid-Margulis}, the random walk with law $\mu$ converges almost surely to a unique point on the Furstenberg boundary $X=G/P$ (here $P$ is the Borel subgroup of $G$), and the corresponding hitting measure (also called \emph{harmonic} or \emph{Furstenberg} measure) $\nu$ is the unique $\mu$-stationary measure on $G/P$.

Brownian motion on $G$ is a continuous analogue of such a random walk, and in that case the analogous hitting measure is of Lebesgue class on $G/P$. If $\Gamma\subseteq G$ is a lattice, Furstenberg has shown that there is a probability measure $\mu$ on $\Gamma$ such that $\nu$ is in the Lebesgue measure class \cite{Furstenberg-BM}, using the  so-called \emph{discretization of Brownian motion}. 

 Let us now consider a random walk on a Zariski-dense group $\Gamma$ with law $\mu$. Due to ergodicity of the Furstenberg measure $\nu$, there is a dichotomy: $\nu$ is either absolutely continuous or singular with respect to Lebesgue measure.   Kaimanovich-Le Prince made the following well-known ``Singularity Conjecture'':
 \begin{conjecture}[{\cite{KlP}}] The Furstenberg measure $\nu$ is singular whenever $\mu$ is finitely supported.\end{conjecture}
 In the above generality, this conjecture is now known to admit counterexamples, with results giving absolutely continuous Furstenberg measures for random walks due to B\'{a}r\'{a}ny-Pollicott-Simon \cite{BPS}, Bourgain \cite{Bourgain}, Benoist-Quint \cite{BQ}, and Lequen \cite{Lequen}. However, in all of these examples, the group $\Gamma_\mu$ is most likely not discrete (in the latter two results, the groups are proven to be dense, whereas in B\'{a}r\'{a}ny-Pollicott-Simon, the groups are generic in a family and Bourgain's groups are generated by small elements).

 However, if $\Gamma$ is discrete, the Singularity Conjecture seems very plausible and is older, see \cite{Kosenko} for a discussion for lattices. There has been much progress on this problem when $G=\textrm{SL}(2,\mathbb{R})$, in which case the conjecture has been proven for noncocompact lattices by Guivarc'h-Le Jan \cite{GlJ} (see also Blach\'ere-Ha\"issinsky-Mathieu \cite{BHM11}) and for some cocompact lattices by Carrasco-Lessa-Paquette \cite{CarrascoLessaPaquette}, Kosenko
 \cite{Kosenko-imrn} and Kosenko-Tiozzo \cite{Kosenko-Tiozzo}. We refer to these references for more information on progress in the setting of rank 1 groups.

 Here, we will start the investigation of this conjecture in higher rank, and we establish the singularity of $\nu$ whenever $\Gamma$ is not a lattice and without any assumptions on $\mu$ (other than that it is symmetric):

\begin{theorem}
\label{thm:main}
Assume that $G$ has Property (T). Let $\Gamma < G$ be a discrete Zariski-dense subgroup that is not a lattice. Let $\mu$ be a symmetric probability measure on $\Gamma$ such that $\textrm{supp} \, \mu$ generates $\Gamma$.

Then the Furstenberg measure $\nu$ on $G/P$ is singular with respect to Lebesgue measure.
\end{theorem}
\begin{remark}\mbox{}
    \begin{enumerate}
        \item This result follows from a more general version that also applies to asymmetric measures: In that case, we prove that at least one of the forward and the backward Furstenberg measure is singular (see Theorem \ref{thm:asymm} for a precise statement).
        \item The only semisimple Lie groups $G$ that do not have Property (T) are those with at least one factor isogenous to $\textrm{SO}(n,1)$ or $\textrm{SU}(n,1)$. In particular, all higher rank simple Lie groups have Property (T).
    \end{enumerate}
\end{remark}

\subsection*{Acknowledgments} While preparing this paper, we discovered that a very similar result to Theorem \ref{thm:main} has independently been obtained by Dongryul Kim and Andrew Zimmer using a different proof \cite{kim-zimmer}. Their theorem does not require $\mu$ to be symmetric. We also thank them for noting an error in an earlier version of this paper (see Remark \ref{rmk:products}). The authors thank Sebastian Hurtado for fruitful discussions.

WvL was supported by NSF-DMS-2203867.
HL was supported in part by AMS-Simons travel grant, a KIAS Individual Grant (HP104101) via the June E Huh Center for Mathematical Challenges at Korea Institute for
Advanced Study, and Sanghyun Kim's 
Mid-Career Researcher Program (RS-2023-00278510) through the National Research Foundation funded by the government of Korea.
\section{Preliminaries}

We need the following measurable analog of proper actions introduced by Margulis \cite{Margulis}:

\begin{definition}
Let $H$ be a locally compact second countable group acting continuously on a locally compact second countable space $X$ with an $H$-quasi-invariant Borel measure $\mu$.  Let $m_H$ be a left Haar measure on $H$. Then the following are all equivalent: 
\begin{enumerate}
\item For $\mu$-almost every $x\in X$, the orbit $Hx$ is closed and the stabilizer $\textrm{Stab}_{H}(x)=\{h\in H: hx=x\}$ is compact in $H$.
\item For $\mu$-almost every $x\in X$, $\textrm{Stab}_{H}(x)$ is compact and the map $A/\textrm{Stab}_{H}(x) \to Hx$ given by $a\textrm{Stab}_{H}(x) \mapsto hx$ is homeomorphism.
\item For $\mu$-almost every $x\in X$, the orbit map $H\to Hx$ given by $h\mapsto hx$ is proper, that is, the set $\{h\in H : hx\in E\}$ is bounded in $H$ for all bounded Borel sets $E\subseteq X$.
\item For $\mu$-almost every $x\in X$ and any bounded Borel subset $E\subset X$, the $m_H$-measure of $\{h\in H : hx \in E\}$ is finite.
\end{enumerate}

If the above conditions are satisfied, then we call the $H$-action \emph{measurably proper} on $(X,\mu)$.
\end{definition}
Margulis discovered a representation-theoretic criterion for measurable properness of $H$-actions:
\begin{theorem}[Margulis \cite{Margulis}]
Let $X=G/\Gamma$ and $m_X$ be a (infinite) Haar measure on $X$. Assume that for every (local) factor $G_i$ of $G$, there are no nonzero $G_i$-invariant vectors in $L^2(G/\Gamma)$. Then the $A$-action on $X$ by left translation is measurably proper. In particular, the $A$-action on $X$ is not ergodic.
\label{thm:measproper}
\end{theorem}
\begin{remark} The above result is actually a combination of several of Margulis' results in \cite{Margulis}: Theorem 2(a) states that if the $G$-action on $G/\Gamma$ is $(G,K,A)$-tempered, then the $A$-action is measurably proper. Remark 1 shows that for temperedness of the action, it suffices that $A$ is $(G,K)$-tempered, i.e. there exists some $L^1$ function on $A$ that bounds the matrix coefficients of $A$ of $K$-invariant vectors in any nontrivial unitary representation of $G$. 

Example (a) claims that $A$ is $(G,K)$-tempered, but unfortunately this is slightly mistaken: The proof here uses Howe-Moore estimates for the matrix coefficients of $A$, but these only apply if there are no nonzero invariant vectors for any local factor of $G$. In the theorem as stated above, we added this additional assumption on the representation, so the Howe-Moore estimates hold and the action of $G$ on $G/\Gamma$ is indeed $(G,K,A)$-tempered.

An earlier version of this paper did not contain the additional assumption. We thank Dongryul Kim and Andrew Zimmer for pointing out that measurable properness does not hold without it, e.g. if $G=G_1\times G_2$ and $\Gamma=\Lambda_1\times \Gamma_2$, where $\Lambda_1\subseteq G_1$ is a lattice, then $A$ does not act measurably properly on $G/\Gamma$.
\label{rmk:products}\end{remark}

Finally, the following lemma is probably well-known, but we could not readily locate a reference, so we include a proof:

\begin{lemma}\label{lem:dich}
Let $\Gamma$ be a countable group of Borel transformations of a Borel space $M$. Suppose $\lambda, \nu$ are two $\Gamma$-quasi-invariant Borel measures on $M$ such that $\nu$ is $\Gamma$-ergodic. Then either $\nu$ is absolutely continuous with respect to $\lambda$ or $\nu$ and $\lambda$ are mutually singular. 
\end{lemma}

\begin{proof}
Suppose that $\nu$ is not absolutely continuous: then there exists a set $A$ with $\nu(A) > 0$ and $\lambda(A) = 0$.
Then $\Gamma A$ is $\Gamma$-invariant, so by ergodicity of $\nu$ one obtains $\nu(\Gamma A) = 1$. On the other hand, 
since $\lambda$ is quasi-invariant one has $\lambda(g A)= 0$ for any $g \in \Gamma$, 
hence $\lambda(\Gamma A) =0$. Thus, $\nu$ and $\lambda$ are singular.
\end{proof}

\section{Proof of the main theorem} \label{sec:proof}
In this section, we retain the setting and notations of \Cref{thm:main}, i.e., let $G$ be a semisimple Lie group with no compact factors, $\Gamma\subseteq G$ be a discrete Zariski-dense subgroup, and $\mu$ a probability measure whose support generates $\Gamma$ as a semigroup. Let $\check{\mu}$  be the reflected measure, i.e. $\check{\mu}(g) := \mu(g^{-1})$. Let $\nu$ and $\check{\nu}$ be the (unique) $(\Gamma, \mu)$- and $(\Gamma, \check{\mu})$-stationary measures on $G/P$, respectively. Let $m_K$ denote the ($K$-invariant) Lebesgue measure on $G/P$.

The Main \Cref{thm:main} is the special case (for symmetric $\mu$) of the following more general result:
\begin{theorem}\label{mainthm2}
\label{thm:asymm}
Assume that $G$ has Property (T). Let $\Gamma < G$ be a discrete Zariski-dense subgroup which is not a lattice. Let $\mu$ be a  probability measure on $\Gamma$ such that $\supp \ \mu$ generates $\Gamma$ as a semi-group. 

Then either $\nu$ or $\check{\nu}$ on $G/P$ is singular with respect to $m_K$.
\end{theorem}

In the remainder of this section, we prove \Cref{mainthm2}.
\subsection{Ergodicity on $G/P\times G/P$}
Let $(B,\delta)$ be the Poisson boundary of $(\Gamma,\mu)$, and let $(B, \check{\delta})$ be the Poisson boundary of $(\Gamma, \check{\mu})$.

We have the following theorem on double-ergodicity.
\begin{theorem}[Kaimanovich \cite{Kaimanovich}]
The diagonal $\Gamma$-action on $(B \times \check{B}, \delta\otimes \check{\delta})$ is ergodic.
\end{theorem}

Now, let $\nu$ denote the unique $\mu$-stationary measure on $G/P$, and let $\check{\nu}$ denote the unique $\check{\mu}$-stationary measure on $G/P$.

\begin{corollary}\label{cor:erg}
The diagonal $\Gamma$-action on $(G/P\times G/P,\nu\otimes \check{\nu})$ is ergodic.
\end{corollary}

\begin{proof}
Since $\Gamma$ is Zariski-dense, the $\mu$-random walk converges almost surely in the Furstenberg compactification, and the hitting measure is supported on $G/P$ \cite{Guivarch-Raugi}. Hence, $(G/P, \nu)$ is a $(\Gamma,\mu)$-boundary, that is, there is a $\Gamma$-equivariant measurable map $(B,\delta)\to (G/P, \nu)$. Since the same is true for the $\check{\mu}$-random walk, 
we obtain a $\Gamma$-equivariant measurable map $\Phi:(B \times B,\delta \otimes \check{\delta})\to (G/P \times G/P, \nu \otimes \check{\nu})$. Hence, 
any $\Gamma$-invariant subset of $G/P \times G/P$ pulls back to a $\Gamma$-invariant subset of $B \times B$.
\end{proof}

\subsection{Non-ergodicity on $G/P\times G/P$}

Now we prove the following non-ergodicity result:
\begin{proposition}\label{prop:nonerg}
Assume that the $A$-action on $G/\Gamma$ is measurably proper. Then the diagonal $\Gamma$-action on $(G/P\times G/P, \lambda_1\otimes \lambda_2)$ is not ergodic for any pair of absolutely continuous $\Gamma$-quasi-invariant measures $\lambda_1$ and $\lambda_2$.
\end{proposition}
\begin{proof} We argue by contradiction so suppose there exist absolutely continuous measures $\lambda_1$ and $\lambda_2$ such that $\Gamma\curvearrowright (G/P\times G/P, \lambda_1\otimes\lambda_2)$ is ergodic.

Write $M:=Z_K(A)$ for the centralizer of $A$ in $K$. It is well-known that with respect to Lebesgue measure, we have $G/P\times G/P\simeq G/MA$ as $G$-spaces: Let $Q$ be the opposite parabolic to $P$, i.e. $Q$ is the image of $P$ under the Cartan involution on $G$. Then $G/P\simeq G/Q$ and $P\cap Q=MA$, and the orbit map $G/MA\to G/P\times G/Q$ is the desired isomorphism.

Therefore we can view $\lambda_1\otimes\lambda_2$ as a measure on $G/MA$. Lift $\lambda_1\otimes\lambda_2$ to a probability measure $\widetilde{\tau}$ on $G$ which is absolutely continuous with respect to Haar along $MA$-orbits and projects to $\lambda_1\otimes\lambda_2$ on $G/MA$:
	$$\widetilde{\tau} := \int_{G/MA} g_\ast \eta_{MA} d(\lambda_1\otimes\lambda_2)(g)$$
where $g$ refers to left-translation by $g$ and $\eta_{MA}$ is a probability measure on $MA$ that is in the measure class of Haar measure (i.e. $d\eta_{MA} = f d m_{MA}$ where $\int_{MA} f dm_{MA} = 1$). Note that $\widetilde{\tau}$ is absolutely continuous with respect to Lebesgue measure.

We claim that $\widetilde{\tau}$ is ergodic for the $\Gamma\times MA$-action on $G$, where $\Gamma$ acts by left-translations and $MA$ acts by right-translations. Indeed, if $U\subseteq G$ is an invariant set, then $U/MA\subseteq G/MA$ is $\Gamma$-invariant and therefore is null or conull with respect to $\lambda_1\otimes\lambda_2$. If $U/MA$ is null, then
	$$\widetilde{\tau}(U)=\int_{U/MA} \eta_{MA}(g^{-1}U\cap MA)  \, d(\lambda_1\otimes\lambda_2)(g) \leq (\lambda_1\otimes\lambda_2)(U/MA)=0.$$
Therefore $U$ is null with respect to $\widetilde{\tau}$. By the same reasoning applied to the complement of $U$, we see that if $U/MA$ is conull with respect to $\lambda_1\otimes\lambda_2$, then $U$ is conull with respect to $\widetilde{\tau}$.

Now we disintegrate $\widetilde{\tau}$ with respect to the projection $G\to \Gamma\backslash G$. We obtain a measure $\tau$ on $\Gamma\backslash G$ that is $MA$-ergodic and absolutely continuous with respect to Lebesgue measure. 

Since $A$ acts measurably properly on $G/\Gamma$ and  $M$ is compact, the $MA$-action on $\Gamma\backslash G$ is also measurably proper. In \cite[p. 455]{Margulis}, Margulis writes that any ergodic measurably proper action is supported on a single closed orbit. This finishes the argument since $\dim(MA)<\dim(G)$, so such measures cannot be absolutely continuous with respect to Lebesgue. For completeness, we include an argument that proves Margulis' claim:

Roughly, for any pair of points $x_1, x_2$ in disjoint $MA$-orbits with the property that any  we want to find small sets $E_1, E_2$ centered at points $x_1,x_2$ in disjoint $MA$-orbits such that $AE_1$ and $AE_2$ are disjoint but each have positive mass with respect to $\tau$. Let us now carry out this construction.

First, let $X'\subseteq \Gamma_0\backslash G$ be the full Lebesgue measure (and hence full $\tau$-measure) locus of points that satisfy the conditions (1) and (2) from the definition of measurable properness of the $MA$-action on $\Gamma\backslash G$. Note that $X'$ is $MA$-invariant.

Let $x_1,x_2\in X'$ be a pair of points that are in distinct $MA$-orbits and with the property that any open neighborhood of them has positive $\tau$-measure. 

Now let $B_i\ni x_i \, (i=1,2)$ be disjoint bounded open neighborhoods of these points such that $MA x_1 \cap \overline{B_2}=\varnothing$ and $\overline{B_1}\cap MAx_2 =\varnothing$. Note that it is possible to choose such sets since the $MA$-orbits of $x_1, x_2$ are closed and disjoint. 

Set $B_i':=B_i\cap X'$. For $x\in B_1'$, let $T_{12}(x)\in [0,\infty]$ such that $ma x \not\in \overline{B_2}$ for any $\|a\|>T_{12}(x)$ and $m\in M$. Note that $T_{12}(x)$ is a.e. finite by Condition (3) of measurable properness and boundedness of $B_2$. 

We now filter $X$ according to the sublevel sets of $T_{12}$, i.e. for $N\geq 0$, let $X_N$ consist of those points $x\in X$ such that $ma x\not\in \overline{B_2}$ for any $m\in M$ and $\|a\|>N$. Since $\{X_N\}_N$ exhaust a full measure set of $X$, we see that $X_N\cap B_1'$ has positive $\tau$-mass for $N\gg 1$. We choose such $N$ and set $E_1:=X_N\cap B_1'$.

We are nearly done, for $MA E_1$ can only intersect $B_2'$ by translates of $ma$ where $\|a\|\leq N$. To also avoid such translates, we shrink $B_2'$ to a smaller set $E_2$ which still has positive $\tau$-mass and such that $max\notin \overline{B}_1$ for any $m\in M$ and $a\in A$ with $\|a\|\leq N$ and $x\in E_2$. We do this as follows: 

For $x\in B_2$, let $R_{21}(x)$ be the smallest absolute value of a time $\|a\|$ such that $m a x \in \overline{B_1}$ for some $m\in M$ and $a\in A$. Note $R_{21}(x_2)=\infty$ since $MAx_2\cap \overline{B_1}=\varnothing$. 

Note that $R_{21}(x)\to \infty$ as $x\to x_2$. For if there were a sequence $y_n\to x_2$ such that $R_{21}(y_n)\leq R<\infty$, then along a subsequence we can assume the corresponding $m_n\to m\in M$ and $R_{21}(y_n)\to r$, and then we must have $m a_r x_2 \in \overline{B_1}$, which is a contradiction.

Now let $E_2\subseteq B_2''$ be sufficiently small that $R_{21}>N$ on $E_2$. Then we have that $MA E_1$ and $MA E_2$ are disjoint, $MA$-invariant, and each carry positive $\tau$-mass. This contradicts ergodicity of $MA$.\end{proof}

In view of Margulis' criterion for measurable properness given in Theorem \ref{thm:measproper}, in order to apply Proposition \ref{prop:nonerg}, we need to show that no local factor of $G$ has nonzero invariant vectors in $L^2(G/\Gamma)$. The following lemma shows that we can pass to a quotient of $G$ where this is true (compare Kim-Zimmer \cite[Proposition 4.1]{kim-zimmer}):

\begin{lemma} There exists a normal subgroup (i.e. collection of local factors) $G_1\subseteq G$ such that the image $\Gamma_0\subseteq G_0:=G/G_1$ is discrete and Zariski-dense, and no local factor of $G_0$ has nonzero invariant vectors in $L^2(G_0/\Gamma_0).$ 
\label{lemma:quotient}
\end{lemma}
\begin{proof} Consider any maximal normal subgroup $G_1$ with a nonzero invariant vector $f\in L^2(G/\Gamma)$. Note that $G_1\neq G$ since otherwise $\Gamma$ is a lattice. Lift $f$ to a function $\widetilde{f}$ on $G$, which is left-invariant under $G_1$ and right-invariant under $\Gamma$. Since $G_1$ is normal, it is also right-invariant under $G_1$ and therefore right-invariant under $G_1\Gamma$, and hence invariant under the closure $\overline{G_1\Gamma}$. 

Let $G_0 = G/G_{1}= G_{1}\backslash G$. The $\overline{G_1 \Gamma}$-invariant function $\widetilde{f}$ on $G$ descends to a function $f_0$ on $G_0=G_{1}\backslash G$ that is invariant under right-translation by $G_1\backslash \overline{G_1\Gamma_0}$. It is straightforward to see that the latter is $\overline{\Gamma_{0}}$ where $\Gamma_{0}$ is the projection of $\Gamma_{0}$ to $G_{1}\backslash G$. Since $\Gamma$ is Zariski-dense in $G$, so is $\Gamma_0$ in $G_0$. Therefore if $\Gamma_0$ is not discrete in $G_0$, then the connected component of $\overline{\Gamma_{0}}$ containing the identity consists of a collection of factors of $G_{0}$, which would contradict the maximality of $G_1$. Further, $\Gamma_0$ could not be a lattice, or averaging $\widetilde{f}$ over $G_0/\Gamma_0$ would give a constant function that descends to a square-integrable function on $G/\Gamma$ (and hence $\Gamma$ would have been a lattice). If local factors of $G_0$ have nonzero invariant functions in $L^2(G_0 /\Gamma_0)$, we repeat this procedure until it terminates. \end{proof}



\begin{proof}[Proof of \Cref{mainthm2}]
Since by \Cref{cor:erg} $\nu \otimes \check{\nu}$ is ergodic, then by \Cref{lem:dich} either $\nu \otimes \check{\nu}$ is absolutely continuous with respect to 
$m_K \otimes m_K$, or 
they are mutually singular.

We aim to rule out the first case: Let $G_1, G_0$ and $\Gamma_0$ be as in Lemma \ref{lemma:quotient}.  We write $\mu_0$ for the push-forward of $\mu$ to $\Gamma_0$. The resulting stationary measures $\nu_0$ and $\check{\nu}_0$ are the pushforwards of $\nu$ and $\check{\nu}$ along the map $G/P\to G_0/P_0$. By Proposition \ref{prop:nonerg},  $\nu_0 \otimes \check{\nu}_0$ and $m_{K_0} \otimes m_{K_0}$ are mutually singular, which implies that either $\nu_0$ or $\check{\nu}_0$ is singular with respect to $m_K$. This shows that either $\nu$ or $\check{\nu}$ is singular with respect to $m_K$.
\end{proof}



\bibliographystyle{alpha}
\bibliography{biblio.bib}

\end{document}